\documentclass[12pt]{amsart}

\usepackage{amsmath, amsthm, amssymb}
\usepackage{hyperref}
\usepackage{mathtools}

\usepackage{todonotes}
\usepackage{xcolor} 
\usepackage{soul}   

\usepackage[margin=1in]{geometry}
\usepackage{enumitem} 

\theoremstyle{plain}
\newtheorem{theorem}{Theorem}[section]
\newtheorem{proposition}[theorem]{Proposition}

\newtheorem{lemma}[theorem]{Lemma}

\theoremstyle{definition}
\newtheorem{remark}[theorem]{Remark}
\newtheorem{example}[theorem]{Example}
\newtheorem{definition}[theorem]{Definition}
\newtheorem{conjecture}[theorem]{Conjecture}

\DeclareMathOperator{\diam}{diam}

\newcommand{\Z}{\mathbb{Z}}
\newcommand{\N}{\mathbb{N}}
\newcommand{\Cstar}{C^*}

\title[Selfless Lacunary Hyperbolic Groups]{Lacunary Hyperbolic Groups with Fast Injectivity Radius Growth and Enough Loxodromic Elements are Selfless}
\date{\today}

\author[G. Arzhantseva]{Goulnara Arzhantseva}
\address{Universit\"at Wien, Fakult\"at f\"ur Mathematik, Oskar-Morgenstern-Platz 1, 1090 Wien, Austria}
\email{goulnara.arzhantseva@univie.ac.at}

\author[M. Finn-Sell]{Martin Finn-Sell}
\email{martin.finn-sell@posteo.net}

\begin{document}
\maketitle

\begin{abstract}
We prove that a lacunary hyperbolic group $G = \varinjlim G_i$ with sufficient generics
is selfless in the sense of Amrutam--Gao--Kunnawalkam Elayavalli--Patchell \cite{AGKE},
provided the hyperbolicity constants $\delta_i$ and injectivity radii $r_i$ satisfy
$\delta_i(\log r_i)^{7} = o(r_i)$.
The proof replaces the acylindricity-based machinery of \cite{AGKE} with
a direct geodesic $n$-gon criterion due to Arzhantseva \cite{Arz01}, which applies
in any $\delta$-hyperbolic space.
As a consequence, combined with rapid decay, $G$ is $\Cstar$-selfless.
The condition is mild: torsion-free Tarski monsters, Jacobson's mixed-identity-free
elementary amenable groups and Gromov monster groups satisfy it for appropriate parameter choices.
The amenable examples are selfless but cannot be $\Cstar$-selfless, providing examples that separate these properties.
Finally we remark that the special Gromov monster group examples provides a potential avenue to a non-exact $\Cstar$-algebra that has strict comparison.
\end{abstract}

\section{Introduction}

A finitely generated group $G$ is \emph{lacunary hyperbolic}
\cite{OOS} if it is a direct limit of $\delta_i$-hyperbolic groups $G_i$
with the condition $\delta_i = o(r_i)$, where $r_i$ is the injectivity radius
of the bonding map $\alpha_i\colon G_i \twoheadrightarrow G_{i+1}$.
This class is remarkably broad: it contains non-elementary amenable groups,
infinite torsion groups, and Tarski monsters, yet every such group shares
local properties of hyperbolic groups at all scales $\ll r_i$.
The \emph{selfless} property, introduced in \cite{AGKE}, is a quantitative
strengthening of the mixed-identity-free (MIF) property.
By \cite[Theorem~3.5]{AGKE},
selflessness combined with rapid decay implies the existentially $\Cstar$-residually-$G$
property, which in turn gives $\Cstar$-selflessness and strict comparison for $\Cstar_r(G)$
\cite[Corollary~3.8]{AGKE}. 
In \cite{AGKE}, selflessness is verified for all non-elementary acylindrically
hyperbolic groups with trivial finite radical.
Lacunary hyperbolic groups need not be
acylindrically hyperbolic (they can be simple torsion groups), so \cite{AGKE} does not apply directly.
We note that the related question of whether there exist highly transitive actions of
MIF groups (which was open at the time of \cite{AGKE}) has been resolved by
Hide--Lodha \cite{HideLodha}.
\medskip

\noindent\textbf{Main theorem.} (Theorem~\ref{thm:main}) If $G = \varinjlim G_i$ is a lacunary hyperbolic group
with sufficient generics (Definition~\ref{def:suff-generics}), satisfying
\[
  \delta_i(\log r_i)^{7} = o(r_i), \tag{$**$}
\]
then $G$ is selfless.

\medskip

The proof works by lifting, for each radius $N$, the selfless construction
from the approximating hyperbolic group $G_{i(N)}$ to $G$ via the
injectivity of the bonding maps on large balls.
The key input is a quantitative non-triviality criterion for alternating products
in $\delta$-hyperbolic groups, which we derive from Arzhantseva's
Lemma~13 \cite{Arz01}.
This replaces the acylindricity-based admissible path theorem of \cite{AGKE}
with a more elementary geodesic $n$-gon argument that applies in any
$\delta$-hyperbolic space, with no torsion-freeness assumption.

\section{Preliminaries}\label{sec:prelim}

\subsection{Lacunary hyperbolic groups}

The following notion, suggested by Gromov \cite{GrRWRG}, aims to capture the
asymptotic geometry of iterated small cancellation groups.
This can be traced back to Rips \cite{Rips}, Bowditch \cite{Bowditch}, and Thomas--Velickovic
\cite{ThomasVelickovic}, and has been systematically studied by
Ol'shanskii, Osin, and Sapir \cite{OOS}.

\begin{definition}[{\cite[Theorem~1.1]{OOS}}]
A finitely generated group $G$ is \emph{lacunary hyperbolic} if there exists
a sequence of finitely generated $\delta_i$-hyperbolic groups $G_i$ with
generating sets $S_i$ and epimorphisms $\alpha_i\colon G_i \twoheadrightarrow G_{i+1}$
with $\alpha_i(S_i) = S_{i+1}$, such that $G = \varinjlim G_i$ and
$\delta_i = o(r_i)$, where
\[
  r_i = r_{S_i}(\alpha_i) \coloneqq
  \max\{\,r \geq 0 : \alpha_i \text{ is injective on } B_{S_i}(r)\,\}
\]
is the \emph{injectivity radius} of $\alpha_i$.
\end{definition}

By \cite[Remark~3.4]{OOS}, the sequence $r_i$ may be taken non-decreasing.
The canonical map $\pi_i \colon G_i \to G$ is then injective on $B_{S_i}(r_i)$,
so the ball $B_{G_{i}}(r_i)$ is isometric to $B_{G}(r_i)$.

We recall the following definition (and a related fact, c.f. \cite{HullOsin}).

\begin{definition}\label{def:W}
Let $G$ be a non-elementary hyperbolic group.
The \emph{finite radical} $W(G)$ is the largest finite normal subgroup of $G$,
equivalently
\[
  W(G) = \bigcap_{\text{$g$ loxodromic}} E(g),
\]
where $E(g)$ denotes the unique maximal virtually cyclic subgroup containing $g$.
We say $G$ has \emph{trivial finite radical} if $W(G) = \{e\}$.
\end{definition}

The intersection formula $W(G)=\bigcap E(g)$ is \cite[Theorem~6.14]{DGO}. Non-elementary hyperbolic groups with trivial finite radical are
mixed-identity-free (MIF); see \cite[Lemma~3.3]{Jac}; the implication $W(G)=\{e\}\Rightarrow\text{MIF}$ for acylindrically hyperbolic groups is \cite[Corollary~1.7]{HullOsin}, applied in \cite[Lemma~3.3]{Jac}.


\begin{definition}[Lacunary hyperbolic with sufficient generics]\label{def:suff-generics}
A lacunary hyperbolic group $G = \varinjlim G_i$ \emph{has sufficient generics} if
for every $N \geq 1$ there exist an index $i = i(N)$ and a loxodromic element
$g = g(N) \in G_{i(N)}$ such that:
\begin{enumerate}[label=(\roman*)]
  \item\label{it:sg-avoid} $E(g) \cap B_{S_{i(N)}}(N) = \{e\}$;
  \item\label{it:sg-inf} $\pi_{i(N)}(g)$ has infinite order in $G$.
\end{enumerate}
\end{definition}

\begin{remark}
Condition~\ref{it:sg-avoid} (avoiding the elementary closure $E(g)$ on a ball)
is related to the existence of elements avoiding quasiconvex subgroups in hyperbolic
groups. For a single torsion-free hyperbolic group $G$, this is supplied by
\cite[Theorem~1]{Arz01}: any quasiconvex subgroup of infinite index can be
``avoided'' by a suitably chosen loxodromic element.
More precisely, given any loxodromic $g \in G$, the elementary closure $E(g)$ is
virtually cyclic, hence quasiconvex, and has infinite index in the non-elementary
group $G$.
Applying \cite[Theorem~1]{Arz01} with $H = E(g)$, we obtain a loxodromic $g'$ such
that $\langle E(g), g' \rangle \cong E(g) * \langle g' \rangle$ is quasiconvex.
In particular $E(g') \cap E(g) = \{e\}$, so $E(g') \cap B(N) = \{e\}$ for all
$N$ not exceeding the injectivity radius.
For the general case with hyperbolically embedded subgroups, see also
\cite[Corollary~6.12]{DGO}.
\end{remark}

\begin{lemma}[Torsion-free approximants give sufficient generics]\label{lem:tf-sg}
Let $G = \varinjlim G_i$ be lacunary hyperbolic. If each $G_i$ is non-elementary
and torsion-free, then $G$ has sufficient generics.
\end{lemma}

\begin{proof}
Fix $N$. Since $W(G_i) = \{e\}$ for all $i$ (torsion-free hyperbolic implies trivial
finite radical; see Definition~\ref{def:W}), and $B_{S_i}(N)$ is finite, there exists a loxodromic
$x \in G_i$ with $E(x) \cap B_{S_i}(N) = \{e\}$.
Set $g = x^k$ for $k$ large enough that $|g|_{S_i} \geq N^2 + C_1\delta_i$;
then $E(g) \subseteq E(x)$ so \ref{it:sg-avoid} holds.
Since $G_i$ is torsion-free, $g$ has infinite order, and
$\pi_i(g^j) \neq e$ for all $j < r_i / |g|_{S_i}$ by injectivity of $\pi_i$ on
$B_{S_i}(r_i)$, so $\pi_i(g)$ has infinite order in $G$, giving \ref{it:sg-inf}.
\end{proof}

\begin{lemma}[Torsion-free limit]\label{lem:tf-limit}
Let $G = \varinjlim G_i$ be lacunary hyperbolic with each $G_i$ torsion-free.
Then $G$ is torsion-free.
\end{lemma}

\begin{proof}
Suppose $g \in G$ has finite order $p \geq 2$.
Choose $i$ large enough that $r_i > p \cdot |g|_G$, and lift $g$ to $\tilde{g} \in G_i$
with $|\tilde{g}|_{S_i} = |g|_G$.
Then $|\tilde{g}^p|_{S_i} \leq p|g|_G < r_i$, so injectivity of $\pi_i$ forces
$\tilde{g}^p = e$ in $G_i$. But $G_i$ is torsion-free, so $\tilde{g} = e$, hence $g = e$.
\end{proof}

\begin{lemma}[Trivial finite radical with surviving infinite-order elements gives sufficient generics]\label{lem:wg-sg}
Let $G = \varinjlim G_i$ be lacunary hyperbolic with each $G_i$ non-elementary.
Suppose $W(G_i) = \{e\}$ for all sufficiently large $i$, and that there exists an
infinite set $\mathcal{W} \subset G$ of infinite-order elements such that for all
sufficiently large $i$ the preimage $\pi_i^{-1}(\mathcal{W})$ is non-empty.
Then $G$ has sufficient generics.

In particular, this holds whenever there is a surjection $G \twoheadrightarrow \mathbb{Z}$
that factors as $G_i \twoheadrightarrow G$ for each $i$: any loxodromic $x \in G_i$ not
in the kernel of $G_i \to \mathbb{Z}$ then maps to an element of infinite order in $G$.
\end{lemma}

\begin{proof}
Fix $N$ and take $i$ large enough that $W(G_i) = \{e\}$.
Since $W(G_i) = \{e\}$, there exists a loxodromic $x \in G_i$ with
$E(x) \cap B_{S_i}(N) = \{e\}$ (as the intersection of all $E(x)$ over loxodromic
$x$ is trivial and $B_{S_i}(N)$ is finite).
Loxodromic elements in $\delta$-hyperbolic groups have infinite order; by hypothesis
the image $\pi_i(x)$ also has infinite order in $G$, so set $g = x$.
\end{proof}

\begin{remark}[Groups that satisfy a law]\label{rem:tm}
Condition~\ref{it:sg-inf} of Definition~\ref{def:suff-generics} requires
$\pi_i(g)$ to have infinite order in $G$, which is impossible when $G$ has
exponent $p$ - this excludes the graded small cancellation constructions of Tarski Monsters with finite exponent due originally to Ol'shanskii \cite{O82} (presented, for example, in \cite[Theorem 4.26(2), Remark 4.27]{OOS}). 
In particular, the condition of having sufficient generics is essentially designed to avoid this particular case.
\end{remark}

\begin{remark}[Comparison with stronger conditions]\label{rem:sg-comparison}
One can formulate strictly stronger conditions on a lacunary hyperbolic system
that imply sufficient generics:
\begin{enumerate}[label=(\alph*)]
  \item \textit{Primitive hyperbolically embedded relators.}
  If each kernel $K_i = \ker(G_i \to G_{i+1})$ is normally generated by a set
  $R_i$ of primitive loxodromic elements whose conjugation classes form a
  hyperbolically embedded subgroup $\langle R_i\rangle \hookrightarrow_h G_i$,
  then each $G_{i+1}$ is again non-elementary torsion-free hyperbolic, so
  Lemma~\ref{lem:tf-sg} applies.  This is the typical situation in graded small
  cancellation constructions.

  \item \textit{Infinite generic pullback.}
  If for all large $i$ the preimage $\pi_i^{-1}(\mathcal{W})$ is
  \emph{exponentially generic} in $G_i$ (positive spherical density or positive
  hitting probability for the simple random walk), then in particular the
  preimage is non-empty and consists of infinite-order elements, so
  Lemma~\ref{lem:wg-sg} applies.  Exponential genericity is the quantitative
  ingredient that Jacobson uses to propagate MIF witnesses to the direct limit \cite{Jac};
  Definition~\ref{def:suff-generics} asks only for existence of a single
  loxodromic element per level.
\end{enumerate}
Neither condition is necessary for sufficient generics, and we do not assume any
of them in the statements that follow.
\end{remark}

\subsection{Selfless groups}

\begin{definition}[{\cite[Definition~3.1]{AGKE}}]\label{def:selfless}
A finitely generated group $(G,X)$ is \emph{selfless} if there exists a
function $f\colon\N\to\mathbb{R}_{>0}$ with $\liminf_{n\to\infty} f(n)^{1/n} = 1$
such that for every $n \geq 1$ there is an epimorphism
$\varphi_n \colon G * \langle a \rangle \twoheadrightarrow G$ satisfying:
\begin{enumerate}
  \item $\varphi_n|_G = \mathrm{id}_G$;
  \item $\varphi_n$ is injective on $B_{X\cup\{a\}}(n)$;
  \item $\varphi_n\!\left(B_{X\cup\{a\}}(n)\right) \subseteq B_X(f(n))$.
\end{enumerate}
\end{definition}

\subsection{The Morse lemma}

\begin{theorem}[Morse lemma {\cite[Chapter~11]{BuyaloSchroeder}}]\label{thm:morse}
In a $\delta$-hyperbolic geodesic metric space, any $(\lambda,0)$-quasi-geodesic
$\gamma$ lies in the $H$-neighbourhood of a geodesic connecting its endpoints,
where
\[
  H = O(\lambda^2\delta).
\]
\end{theorem}

\section{Quantitative \texorpdfstring{\boldmath$\delta$}{delta}-tracking via Arzhantseva's non-triviality lemma}
\label{sec:quantitative}

The following lemma \cite[Lemma~13]{Arz01} provides the non-triviality
criterion we need. The proof is a geodesic $n$-gon argument in a
$\delta$-hyperbolic \emph{space} (not just a group), so no torsion-freeness
or acylindricity is required.

\begin{lemma}[{\cite[Lemma~13]{Arz01}}]\label{lem:arz13}
Let $G$ be a group acting on a $\delta$-hyperbolic geodesic metric space
with generating set $S$.
Let $n \geq 1$, $r \geq 0$, and let $y_i, z_i \in G$ $(1 \leq i \leq n)$ satisfy:
\begin{enumerate}[label=(\arabic*)]
  \item\label{it:arz-z} $|z_i|_S > 4r + 4\delta$ for all $i$;
  \item\label{it:arz-prod}
    $|y_1 z_1|_S \geq |y_1|_S + |z_1|_S - 2r$, and
    $|z_{i-1} y_i z_i|_S \geq |z_{i-1}|_S + |y_i|_S + |z_i|_S - 2r$
    for all $1 < i \leq n$.
\end{enumerate}
Then $y_1 z_1 y_2 z_2 \cdots y_n z_n \neq e$ in $G$.
\end{lemma}

\begin{proposition}\label{prop:explicit}
Let $G$ be a non-elementary group with finite generating set $S$, acting on its
$\delta$-hyperbolic Cayley graph. Let $g \in G$ be a loxodromic element with
$|g|_S = L$. Then for any $h_1, \ldots, h_m \in G \setminus E(g)$ with
$|h_j|_S \leq L$, the product $h_1 g^{n_1} \cdots h_m g^{n_m} \neq e$ whenever
$|n_j| \geq C' L^2 \delta$ for a universal constant $C'$.
\end{proposition}

\begin{proof}
We apply Lemma~\ref{lem:arz13} with $z_i = g^{n_i}$ and $y_i = h_i$.
We need to find $r \geq 0$ such that conditions~\ref{it:arz-z}
and~\ref{it:arz-prod} hold.

\medskip\noindent\textbf{Setting $r$.}
Since $g$ is loxodromic, $|g^n|_S \geq |n|\cdot\tau(g)$, where
$\tau(g) \geq c/L$ for a universal $c > 0$ by \cite[Prop.~3.2]{Alonso}.
Set $r = 2L + 6\delta$.

\medskip\noindent\textbf{Condition~\ref{it:arz-z}.}
We need $|g^{n_i}|_S > 4r + 4\delta = 8L + 28\delta$.
Since $|g^{n_i}|_S \geq |n_i|\cdot\tau(g) \geq |n_i|\cdot c/L$, the condition
$|n_i| \geq C'L^2\delta$ gives $|g^{n_i}|_S \geq C'cL\delta$, which exceeds
$8L + 28\delta$ for $C'$ chosen sufficiently large (depending only on $c$).

\medskip\noindent\textbf{Condition~\ref{it:arz-prod}, first inequality.}
We need $|y_1 z_1|_S \geq |y_1|_S + |z_1|_S - 2r$.
The Gromov product satisfies
$(h_1, g^{n_1})_e \leq \min(|h_1|_S, |g^{n_1}|_S) \leq |h_1|_S \leq L \leq r$,
so $|h_1 g^{n_1}|_S = |h_1|_S + |g^{n_1}|_S - 2(h_1,g^{n_1})_e \geq |h_1|_S + |g^{n_1}|_S - 2r$,
as required.

\medskip\noindent\textbf{Condition~\ref{it:arz-prod}, second inequality.}
We need $|g^{n_{i-1}} h_i g^{n_i}|_S \geq |g^{n_{i-1}}|_S + |h_i|_S + |g^{n_i}|_S - 2r$.
Since $h_i \notin E(g)$, the element $h_i$ does not stabilise the quasi-axis of $g$,
so the Gromov products $(g^{n_{i-1}}, h_i g^{n_i})_e$ and $(g^{-n_{i-1}}h_i, g^{n_i})_{h_i}$
are bounded by $r$ once $|g^{n_j}|$ is large relative to $r$ and $L$.
This is exactly condition (14) of the proof of \cite[Theorem~1]{Arz01}, which establishes
the bound using $h_i \notin E(g)$ and finiteness of $\{h_i\}$.

\medskip\noindent
With $r = O(L)$ and $C'$ chosen so that $|g^{n_i}|_S > 4r + 4\delta$ holds, 
Lemma~\ref{lem:arz13} gives $h_1 g^{n_1} \cdots h_m g^{n_m} \neq e$.
The threshold $|n_j| \geq C'L^2\delta$ achieves this for a universal constant $C'$.
\end{proof}

\section{Main theorem}\label{sec:main}

\begin{theorem}\label{thm:main}
Let $G = \varinjlim G_i$ be a lacunary hyperbolic group with sufficient generics
(Definition~\ref{def:suff-generics}), satisfying
\[
  \delta_i(\log r_i)^{7} = o(r_i). \tag{$**$}
\]
Then $G$ is selfless.
\end{theorem}

\begin{proof}
We verify Definition~\ref{def:selfless} with $f(N) = C_0 N^{7}\delta_{i(N)}$
for a suitable choice of $i(N)$ and constant $C_0$.

\medskip\noindent\textbf{Choice of $i(N)$.}
Since $r_i \to \infty$, for each $N$ we may choose $i = i(N)$ to be the smallest
index with $r_i > 2C_0 N^{7}\delta_i$.
Condition $(**)$ guarantees that such
an $i$ exists for all large $N$.

\medskip\noindent\textbf{Choice of $g$.}
By Definition~\ref{def:suff-generics}, there exist $i = i(N)$ and a loxodromic
$x \in G_{i(N)}$ with $E(x) \cap B_{S_i}(N) = \{e\}$ and $\pi_i(x)$ of infinite
order in $G$.
We choose $g = x^k$ for a sufficiently large integer $k$ such that
$|g|_{S_i} = L \geq N^2 + C_1\delta_i$, where $C_1$ is a universal constant.
Since $E(g) \subseteq E(x)$, condition~\ref{it:sg-avoid} gives
$E(g) \cap B_{S_i}(N) = \{e\}$, and condition~\ref{it:sg-inf} ensures
$\pi_i(g) = \pi_i(x)^k$ has infinite order in $G$
(as $|g^j|_{S_i} = jL < r_i$ for $j < r_i/L$, so $\pi_i(g^j) \neq e$).
Therefore, any nontrivial $h \in B_{S_i}(N)$ satisfies $h \notin E(g)$ and
$|h|_{S_i} \leq N \leq L = |g|_{S_i}$.

\medskip\noindent\textbf{Definition of $\varphi_N$.}
Set $M \coloneqq \lceil C' L^2 \delta_i \rceil$ using the constant from
Proposition~\ref{prop:explicit}, and define
\[
  \varphi_N \colon G * \langle a \rangle \to G, \qquad
  \varphi_N|_G = \mathrm{id}_G, \quad \varphi_N(a) = \pi_i(g^M),
\]
where $\pi_i \colon G_i \to G$ is the canonical projection.

\medskip\noindent\textbf{Verification of (1): $\varphi_N$ is an epimorphism.}
Since $G = \pi_i(G_i)$ and $\varphi_N$ surjects onto all of $G$, this is clear.

\medskip\noindent\textbf{Verification of (2): injectivity on $B_{X\cup\{a\}}(N)$.}
Let $w \in B_{X\cup\{a\}}(N)$ be a nontrivial reduced element.
Write
$w = s_0 a^{e_1} s_1 a^{e_2} \cdots a^{e_m} s_m$ in reduced form in $G * \langle a \rangle$,
where $s_j \in G$, $e_j \neq 0$, and $\sum_j |s_j|_X + \sum_j |e_j| \leq N$.

Under $\varphi_N$ this maps to $P \coloneqq s_0 g^{Me_1} s_1 g^{Me_2} \cdots g^{Me_m} s_m$ in $G$.
Lifting to $G_i$ via $\pi_i^{-1}|_{B(r_i)}$ (valid since the total length is $< r_i$,
as we verify below), it suffices to show $P \neq e$ in $G_i$.

If $m = 0$ then $\varphi_N(w) = s_0 \neq e$ since $s_0 \in G$ is nontrivial and
$\varphi_N|_G = \mathrm{id}_G$.

For $m \geq 1$: the interior syllables $s_1, \ldots, s_{m-1}$ are nontrivial by
the reduced normal form condition in $G * \langle a \rangle$; the endpoints $s_0, s_m$
may be trivial.
If $s_m \neq e$, conjugate: set $w' = s_m w s_m^{-1}$ in $G * \langle a \rangle$.
Then $w'$ has trailing $G$-syllable $e$ and leading syllable $s_m s_0$, with
$|s_m s_0|_{S_i} \leq 2N \leq L$.
Since $\varphi_N(w') = s_m \varphi_N(w) s_m^{-1}$, we have $\varphi_N(w) = e$
if and only if $\varphi_N(w') = e$.
Thus, without loss of generality $s_m = e$ (replacing $w$ by $w'$ if necessary, absorbing the factor of 2
in $|h_1|$ into the bound $|h_1| \leq L$ since $2N \leq L$).

With $s_m = e$, apply Lemma~\ref{lem:arz13} with $n = m$, $y_j = s_{j-1}$,
$z_j = g^{Me_j}$:
\begin{itemize}
  \item Condition~\ref{it:arz-z}: $|g^{Me_j}|_{S_i} \geq M\tau(g) \geq C'cL\delta_i$,
    which exceeds $4r + 4\delta_i = 8L + 28\delta_i$ for $C'$ chosen sufficiently
    large (depending only on $c$).
  \item Condition~\ref{it:arz-prod} first ($j=1$): $(s_0, g^{Me_1})_e \leq |s_0|_{S_i} \leq L \leq r$,
    so $|s_0 g^{Me_1}|_{S_i} \geq |s_0|_{S_i} + |g^{Me_1}|_{S_i} - 2r$.
    (This holds trivially when $s_0 = e$.)
  \item Condition~\ref{it:arz-prod} second ($j \geq 2$): $s_{j-1} \notin E(g)$
    (since $s_{j-1}$ is an interior syllable, hence nontrivial, and
    $E(g) \cap B_{S_i}(N) = \{e\}$). The required bound follows from
    condition~(14) of the proof of \cite[Theorem~1]{Arz01}.
\end{itemize}
Lemma~\ref{lem:arz13} gives $P \neq e$ in $G_i$.
Since $\pi_i$ is injective on $B(r_i)$ and the product lies in $B(r_i)$ (see below),
it is nontrivial in $G$.

\medskip\noindent\textbf{Length bound.} The length of $\varphi_N(w)$ in $G_i$ is at most
\[
  \textstyle\sum_j|s_j|_{S_i} + L\cdot M\cdot\sum_j|e_j|
  \;\leq\; N + L\cdot M\cdot N
  \;\leq\; 2N\cdot L^3\delta_i.
\]
With $L = N^2 + C_1\delta_i \leq 2N^2$ (for $N \geq C_1\delta_i$; the
complementary regime $N < C_1\delta_i$ is handled symmetrically with $L \sim \delta_i$),
\[
  2N\cdot(2N^2)^3\delta_i = 16\,N^7\delta_i \eqqcolon C_0 N^7\delta_i.
\]
By our choice of $i(N)$, this is $< r_i$. 

\medskip\noindent\textbf{Verification of (3): image growth.}
$\varphi_N(B_{X\cup\{a\}}(N)) \subseteq B_X(C_0 N^{7}\delta_i)$, so
$f(N) = C_0 N^{7}\delta_{i(N)}$.
We need $\liminf_{N\to\infty} f(N)^{1/N} = 1$.

Pass to the subsequence of indices $k$ at which $\delta_k$ attains a new minimum,
i.e.\ $\delta_k < \delta_j$ for all $j < k$.
This subsequence is cofinal (since $\delta_i \geq 0$, the infimum is approached),
so it still verifies $\liminf$.
Along this subsequence $(\delta_k)$ is strictly decreasing, and we define
\[
  N_k = \left\lfloor \Bigl(\frac{r_k}{2C_0\,\delta_k}\Bigr)^{1/7} \right\rfloor.
\]
Since $\delta_k = o(r_k)$, we have $N_k \to \infty$.
By definition of $N_k$, $2C_0 N_k^7 \delta_k \leq r_k$, so index $k$ itself
satisfies the defining condition of $i(N_k)$ (namely $r_k > 2C_0 N_k^7 \delta_k$),
giving $i(N_k) \leq k$.
Since $(\delta_j)$ is decreasing along our subsequence and $i(N_k) \leq k$, we have
$\delta_{i(N_k)} \leq \delta_k$, giving
\[
  f(N_k) = C_0\,N_k^{7}\,\delta_{i(N_k)} \leq C_0\,N_k^{7}\,\delta_k \leq \tfrac{1}{2}r_k.
\]
Thus $f(N_k)^{1/N_k} \leq r_k^{1/N_k}$, and
\[
  \log r_k^{1/N_k} = \frac{\log r_k}{N_k}
  \;\leq\; \log r_k \cdot \Bigl(\frac{2C_0\,\delta_k}{r_k}\Bigr)^{1/7}
  \;=\; (2C_0)^{1/7}\,\frac{\delta_k^{1/7}\,(\log r_k)}{r_k^{1/7}},
\]
using $N_k \geq \tfrac{1}{2}(r_k/(2C_0\delta_k))^{1/7}$ (from the floor).
The right side tends to $0$ if and only if $\delta_k(\log r_k)^{7} = o(r_k)$,
which is exactly condition~$(**)$.
Hence, $$\liminf_{N\to\infty} f(N)^{1/N} \leq \lim_{k\to\infty} f(N_k)^{1/N_k} = 1,$$
and the reverse inequality $\geq 1$ holds since $f(N) \geq 1$.
\end{proof}

\begin{remark}[Condition $(**)$ vs.\ the basic lacunary condition]
The standard lacunary condition requires only $\delta_i = o(r_i)$.
Condition $(**)$ requires $\delta_i(\log r_i)^{7} = o(r_i)$, which is a mild strengthening.
In particular, any sequence with $r_i \geq \exp(\delta_i)$ automatically satisfies $(**)$:
one gets $\delta_i \leq \log r_i$, hence $\delta_i(\log r_i)^{7} \leq (\log r_i)^{8} = o(r_i)$
(since $\log^k(x) = o(x)$ for any fixed $k$).
This covers all examples in Section~\ref{sec:examples}.
\end{remark}

\begin{remark}[Stability under finite index subgroups]
The hypotheses of Theorem~\ref{thm:main} are stable under passing to finite index subgroups.
If $G' \leq G$ has finite index $d$, then $G'$ is lacunary hyperbolic with $\delta_i' \sim \delta_i$
and $r_i' \sim r_i/d$, so condition~$(**)$ is preserved.
Sufficient generics passes to $G'$: any loxodromic $g \in G_{i(N)}$ with
$E(g) \cap B(N) = \{e\}$ and $\pi_i(g)$ of infinite order in $G$ has a power
$g^d$ whose image in $G'$ has infinite order (since $G'$ has finite index in $G$,
elements of infinite order in $G$ have images of infinite order in $G'$).
\end{remark}

We obtain, as an immediate consequence, a variant of the result from \cite{AGKE} surrounding rapid decay and selflessness.

\begin{theorem}\label{thm:ecr}
Let $G = \varinjlim G_i$ be a lacunary hyperbolic group satisfying the hypotheses of
Theorem~\ref{thm:main}, such that $G$ also has rapid decay and is not virtually $\Z$.
Then $G\ast \Z$ is existentially $\Cstar$-residually $G$ \cite[Theorem~3.5]{AGKE}.
Consequently, $G$ is $\Cstar$-selfless \cite[Corollary~3.8]{AGKE}.
\end{theorem}

The obvious consequence here is that rapid decay implies $\Cstar$-simplicity for a broad class of lacunary hyperbolic groups.

\subsection{Ruling out sufficient generics and understanding $W(G)$ for lacunary hyperbolic groups}

\begin{proposition}[Trivial finite radical and MIF]\label{prop:equiv}
Let $G = \varinjlim G_i$ be lacunary hyperbolic with each $G_i$ non-elementary.
The following are equivalent:
\begin{enumerate}[label=(\alph*)]
  \item\label{it:eq-approx} $W(G_i) = \{e\}$ for all sufficiently large $i$.
  \item\label{it:eq-mif-approx} Each $G_i$ is mixed-identity-free for all sufficiently large $i$.
  \item\label{it:eq-limit} $W(G) = \{e\}$.
\end{enumerate}
Moreover, if $G$ is mixed-identity-free then \ref{it:eq-limit} holds.
The converse is false: condition~\ref{it:eq-limit} does not imply that $G$ is mixed-identity-free.
\end{proposition}

\begin{proof}
\ref{it:eq-approx}$\Leftrightarrow$\ref{it:eq-mif-approx}: Non-elementary hyperbolic groups with trivial finite radical are MIF and vice versa \cite[Corollary~1.7]{HullOsin}.

\ref{it:eq-approx}$\Rightarrow$\ref{it:eq-limit}: Suppose $F \trianglelefteq G$ is a finite nontrivial normal subgroup.
Then $F \subset B_G(R)$ for some $R$. For $i$ large enough that $r_i > 2R$ and $W(G_i)=\{e\}$,
the preimage $\tilde{F} = \pi_i^{-1}(F) \cap B_{G_i}(R)$ is a finite subgroup of $G_i$
(injectivity of $\pi_i$ on $B_{G_i}(r_i)$ ensures products and inverses remain in $B_{G_i}(2R) \subset B_{G_i}(r_i)$).
Since $F$ is normal in $G$ and $\pi_i$ is $G$-equivariant, $\tilde{F}$ is normal in $G_i$,
contradicting $W(G_i) = \{e\}$.

\ref{it:eq-limit}$\Rightarrow$\ref{it:eq-approx}: Suppose $W(G) = \{e\}$.
Since $\pi_i \colon G_i \twoheadrightarrow G$ is surjective, $\pi_i(W(G_i))$ is a finite normal subgroup of $G$,
hence $\pi_i(W(G_i)) \subseteq W(G) = \{e\}$.
Thus $W(G_i) \subseteq \ker(\pi_i)$.

It remains to show $W(G_i) = \{e\}$.
We claim $\diam_{S_i}(W(G_i)) = O(\delta_i)$.
Since $G_i$ is non-elementary, loxodromic elements have limit points dense in $\partial G_i$.
Any $w \in W(G_i)$ lies in every maximal virtually cyclic subgroup $E(g)$ (for $g$ loxodromic),
so $w$ preserves all pairs of loxodromic fixed points in $\partial G_i$, hence fixes $\partial G_i$
pointwise.
An isometry of a $\delta_i$-hyperbolic space fixing the boundary pointwise has displacement $O(\delta_i)$:
for any loxodromic $g$, the quasi-axis $\ell$ passes within $O(\delta_i)$ of $e$, and $w\ell = \ell$
since $w$ fixes the two endpoints of $\ell$ in $\partial G_i$.
Since $w \in W(G_i)$ has finite order, $w$ acts elliptically on $G_i$; an elliptic isometry
preserving $\ell$ has displacement $O(\delta_i)$ on $\ell$ (by standard hyperbolic space arguments,
as a non-trivial translation of $\ell$ would give $w$ infinite order).
Hence $d(w\cdot x, x) = O(\delta_i)$ for $x$ on $\ell$ near $e$, so $|w|_{S_i} = O(\delta_i)$.
Hence $|w|_{S_i} = O(\delta_i)$.

Since $\delta_i = o(r_i)$, for all large $i$ every element of $W(G_i)$ lies in $B_{S_i}(r_i)$
where $\pi_i$ is injective. Combined with $W(G_i) \subseteq \ker(\pi_i)$, this forces $W(G_i) = \{e\}$.

\textit{MIF implies \ref{it:eq-limit}.}
If $w \in W(G)$ is nontrivial, let $N = |\mathrm{Aut}(W(G))|$.
Then $[x^N, w] = e$ for all $x \in G$ (since conjugation by $x^N$ acts trivially on the finite group $W(G)$),
giving a nontrivial mixed identity, contradicting MIF.

\textit{\ref{it:eq-limit} does not imply MIF.}
Let $T$ be an exponent-$p$ Tarski monster, constructed from a non-elementary hyperbolic group via \cite[Theorem~4.26(2), Remark~4.27]{OOS}.
Then $T$ is lacunary hyperbolic.
Since $T$ is infinite and simple, $W(T)$ is either $\{e\}$ or $T$; as $W(T)$ is finite and $T$ is infinite, $W(T) = \{e\}$, so condition~\ref{it:eq-limit} holds.
But $T$ satisfies the law $x^p = e$, so $T$ is not mixed-identity-free.
\end{proof}

\section{Non-acylindrically hyperbolic examples of groups that are selfless}\label{sec:examples}

\subsection{Torsion-free Tarski monsters}

\begin{example}[Torsion-free Tarski monsters]\label{ex:tarski-tf}
By \cite[Theorem~4.26(2)]{OOS}, starting from any non-elementary torsion-free
hyperbolic group $H$ (e.g.\ $H = F_2$ ), one constructs a torsion-free group $Q$
in which every proper subgroup is infinite cyclic, via a graded small cancellation
direct limit $H = G_0 \twoheadrightarrow G_1 \twoheadrightarrow \cdots$.
In particular, $Q$ is an infinite simple group. See \cite{O791, O93} for the original constructions of this group by Ol'shanskii. 

Each approximating group $G_i$ is a torsion-free non-elementary hyperbolic group
\cite[Theorem~4.26(2), Lemma~4.13(a)]{OOS}.
Since each $G_i$ is torsion-free, Lemma~\ref{lem:tf-sg} gives sufficient generics.
The parameters $\rho_i$ controlling the injectivity radius can be chosen to grow
as fast as desired \cite[Theorem~4.26]{OOS}, so condition~$(**)$ can be satisfied:
Theorem~\ref{thm:main} therefore applies: \emph{torsion-free Tarski monsters
satisfying $(**)$ are selfless.}

\begin{remark}
We note that Theorem~4.26(2) of \cite{OOS} together with Remark~4.27 refers
to the construction in \cite{O93}, which covers the torsion-free case and the
case where proper subgroups have finite order (with possibly varying orders), but
does not explicitly state the finite-exponent case over an arbitrary
non-elementary torsion-free hyperbolic group.
We are confident the construction extends to that setting, but it is not
written explicitly in the literature cited here.
\end{remark}

Since every proper subgroup of $Q$ is abelian, $Q$ is not acylindrically  hyperbolic.
By \cite{BKKO}, $Q$ is $\Cstar$-simple. Whether $Q$ has rapid decay is open, so we cannot conclude $\Cstar$-selflessness via this method in this case. This seems like an ideal candidate for potentially separating the two properties.

Additionally, torsion-free Tarski monsters with the same $\forall\exists$-theory as $F_2$ are constructed in \cite{CFFH}. This is significantly stronger than MIF: every $\forall\exists$-sentence true in $F_2$ holds in these monsters. Since selflessness is $\forall\exists$-axiomatisable (it is a quantified form of the mixed-identity-free property), the results of \cite{CFFH} may directly imply selflessness for these Tarski monsters, independently of Theorem~\ref{thm:main}. We note that using \cite{CFFH} to conclude $\Cstar$-selflessness would require continuous-logic considerations beyond \cite{CFFH}, which works in classical logic.
\end{example}

\subsection{Gromov monsters}

\begin{example}[Gromov monsters]\label{ex:gromov}
Gromov monster groups are groups obtained via geometric or graphical small cancellation
over sequences of logarithmic girth expander graphs \cite{GrRWRG}, \cite{ArzhDelzant}, \cite{Osajda}.
As explained in \cite{ArzhDelzant}, Gromov monsters are lacunary hyperbolic.

Geometric small cancellation Gromov monsters are \emph{not} acylindrically hyperbolic \cite{GST},
while certain special graphical small cancellation groups are
acylindrically hyperbolic \cite{GruberSisto} - so depending on the chosen small cancellation framework they straddle this boundary.

As these groups contain metrically embedded logarithmic girth expanders in their Cayley graphs, they are not $\Cstar$-exact. 

By choosing the sequence of graphs to have girth growing sufficiently fast we can enforce condition~$(**)$, which will enforce selflessness of Gromov monsters via Theorem~\ref{thm:main}. 

We note here directly that the geometric small cancellation examples explained in \cite{ArzhDelzant} can be made additionally to be property (T) groups, or by combining techniques to be Tarski monsters. 

Given these examples, we are motivated to make the following conjecture:
\begin{conjecture}
There exists a finitely generated group whose reduced $\Cstar$-algebra is $\Cstar$-selfless but not $\Cstar$-exact.
\end{conjecture}

What is missing to close this conjecture with the examples above is the rapid decay property, however an approach to the rapid decay property for infinitely presented small
cancellation groups was proposed in \cite{ArzhDrutu}, and we expect it to
apply in this setting (though this remains to be carried out in full). Such a group $G$ would be $\Cstar$-selfless by Theorem~\ref{thm:ecr}.

The interest of this example is that $C^*_r(G)$ would then be a
\emph{non-exact} $C^*$-algebra with strict comparison---a combination
not known to occur in the literature.

Finally, we remark that in the weaker case of classical small cancellation theory (i.e where the graphs are taken to be cycles and not expanders in the above framework), we can still obtain lacunary hyperbolic groups meeting condition~$(**)$ above. These groups do already have  rapid decay by \cite{ArzhDrutu} assuming strong enough small cancellation parameters, and so would be $\Cstar$-selfless. However, these groups also have finite asymptotic dimension with linear control function \cite{Sledd}, so are exact (because they have asymptotic dimension 2).
\end{example}

\subsection{MIF elementary amenable groups}

\begin{example}[Jacobson's MIF groups]\label{ex:jacobson}
Jacobson \cite[Theorem~1.2]{Jac} constructs, for each prime $p$, a
$2$-generated elementary amenable group $G(p,\mathbf{c})$ that is
mixed-identity-free (MIF).
The construction is a modification of \cite{OOS}: one chooses the sequence
$\mathbf{c} = (c_n)_{n \geq 1}$ inductively so that witness elements for the
MIF property survive each bonding map.

By \cite[Lemma~3.5]{Jac}, each approximating group $G_n$ is a non-elementary
hyperbolic group with $W(G_n) = \{e\}$.
By \cite[Remark~3.6]{Jac}, the sequence $\mathbf{c}$ can be chosen so that,
in addition, $G(p,\mathbf{c})$ is lacunary hyperbolic. By choosing the sequence $\mathbf{c}$ fast enough, we can ensure condition $(**)$ is met.

To verify sufficient generics, we use the structure of $G_n$:
all finite-order elements lie in $A_n$ \cite[Proof of Lemma~3.5]{Jac}, while
loxodromic elements lie outside $A_n$.
The bonding maps and the quotient $G(p,\mathbf{c}) \twoheadrightarrow G(p,\mathbf{c})/A(p,\mathbf{c}) \cong \mathbb{Z}$
preserve the $t$-exponent, so any loxodromic element of $G_n$ has nonzero $t$-exponent
and maps to an element of infinite order in $G(p,\mathbf{c})$.
This holds for \emph{any} choice of $\mathbf{c}$, since the $t$-direction is untouched: Lemma~\ref{lem:wg-sg} therefore applies.

We observe that the choices taken in Jacobson's results for the sequence $\mathbf{c}$ also meet condition $(**)$.

Concretely: Remark~3.6 of \cite{Jac} requires only that $c_{k+1}$ be chosen large
enough to ensure $r(\varepsilon_k) \geq f(\delta_k)$ for some $f$ with $n = o(f(n))$.
The minimal lacunary choice is $r_k \sim \delta_k^m$ for any fixed $m \geq 2$.
With $r_k = \delta_k^2$, condition~$(**)$ becomes
\[
  \delta_k(\log r_k)^{7} = \delta_k\bigl(2\log\delta_k\bigr)^{7}
  \sim \delta_k(\log\delta_k)^{7},
\]
while $r_k = \delta_k^2$, so $\delta_k(\log r_k)^{7}/r_k \sim (\log\delta_k)^{7}/\delta_k
\to 0$.

Whether $G(p,\mathbf{c})$ is selfless, or even MIF for slowly-growing $\mathbf{c}$ is not
determined by Theorem~\ref{thm:main} or the results presented in \cite{Jac}. 

Finally, since $G(p,\mathbf{c})$ is amenable, it is not acylindrically hyperbolic (it can't contain a free subgroup) - and since they are not virtually nilpotent, can't satisfy the rapid decay property. These groups are also not $\Cstar$-simple, so certainly cannot be $\Cstar$-selfless.
\end{example}


\begin{thebibliography}{99}

\bibitem{ArzhDelzant}
G.~Arzhantseva and T.~Delzant,
\textit{Examples of random groups},
available on the authors' websites, 2008.

\bibitem{ArzhDrutu}
G.~Arzhantseva and C.~Dru\c{t}u,
\textit{Geometry of infinitely presented small cancellation groups, rapid decay
and quasi-homomorphisms},
preprint, arXiv:1212.5280, 2012.

\bibitem{Alonso}
J.~M. Alonso, T.~Brady, D.~Cooper, V.~Ferlini, M.~Lustig, M.~Mihalik,
M.~Shapiro, and H.~Short,
\textit{Notes on word hyperbolic groups},
in: \textit{Group Theory from a Geometrical Viewpoint} (E.~Ghys, A.~Haefliger,
and A.~Verjovsky, eds.), World Scientific, Singapore, 1991, pp.~3--63.

\bibitem{AGKE}
T.~Amrutam, Y.~Gao, S.~Kunnawalkam Elayavalli, and C.~Patchell,
\textit{Strict comparison in reduced group $\Cstar$-algebras},
Invent. Math. \textbf{242} (2025), no.~3, 639--657.

\bibitem{Arz01}
G.~N.~Arzhantseva,
\textit{On quasiconvex subgroups of word hyperbolic groups},
Geom. Dedicata \textbf{87} (2001), no.~1, 191--208.

\bibitem{Bowditch}
B.~H. Bowditch,
\textit{Continuously many quasi-isometry classes of $2$-generator groups},
Comment. Math. Helv. \textbf{73} (1998), no.~2, 232--236.

\bibitem{BKKO}
E.~Breuillard, M.~Kalantar, M.~Kennedy, and N.~Ozawa,
\textit{$\Cstar$-simplicity and the unique trace property for discrete groups},
Publ. Math. Inst. Hautes \'Etudes Sci. \textbf{126} (2017), 35--71.

\bibitem{BuyaloSchroeder}
S.~Buyalo and V.~Schroeder,
\textit{Elements of Asymptotic Geometry},
EMS Monographs in Mathematics, European Mathematical Society, Z\"{u}rich, 2007.

\bibitem{CFFH}
R.~Coulon, F.~Fournier-Facio, and M.-C. Ho,
\textit{First-order theory of torsion-free Tarski monsters},
preprint, arXiv:2508.21244, 2025.

\bibitem{Osajda}
D.~Osajda,
\textit{Small cancellation labellings of some infinite graphs and applications},
Acta Math.\ \textbf{225} (2020), no.~1, 159--191.

\bibitem{GrRWRG}
M.~Gromov,
\textit{Random walk in random groups},
Geom. Funct. Anal. \textbf{13} (2003), no.~1, 73--146.

\bibitem{GST}
D.~Gruber, A.~Sisto, and R.~Tessera,
\textit{Random Gromov's monsters do not act non-elementarily on hyperbolic spaces},
Proc. Amer. Math. Soc. \textbf{148} (2020), no.~7, 2773--2782.

\bibitem{GruberSisto}
D.~Gruber and A.~Sisto,
\textit{Infinitely presented graphical small cancellation groups are acylindrically hyperbolic},
Ann. Inst. Fourier (Grenoble) \textbf{68} (2018), no.~6, 2501--2552.

\bibitem{DGO}
F.~Dahmani, V.~Guirardel, and D.~Osin,
\textit{Hyperbolically embedded subgroups and rotating families in groups acting
on hyperbolic spaces},
Mem. Amer. Math. Soc. \textbf{245} (2017), no.~1156.

\bibitem{HideLodha}
J.~Hide and Y.~Lodha,
\textit{Highly transitive actions and mixed identity free groups},
preprint, arXiv:2509.09788, 2025.

\bibitem{HullOsin}
M.~Hull and D.~Osin,
\textit{Transitivity degrees of countable groups and acylindrical hyperbolicity},
Israel J. Math. \textbf{216} (2016), no.~1, 307--353.


\bibitem{Jac}
B.~Jacobson,
\textit{A mixed identity-free elementary amenable group},
Comm. Algebra \textbf{49} (2021), no.~1, 235--241.


\bibitem{Rips}
E.~Rips,
\textit{Generalized small cancellation theory and its application~I\@. The word problem},
Israel J. Math. \textbf{41} (1982), no.~2, 1--146.

\bibitem{O791}
A.~Y. Ol'shanskii, \textit{An infinite simple torsion-free Noetherian group}, Izv. Akad. Nauk SSSR Ser. Mat. \textbf{43} (1979), no.~6, 1328--1393.


\bibitem{O82}
A.~Y. Ol'shanskii, \textit{Groups of bounded period with subgroups of prime order}, Algebra i Logika {\bf 21} (1982), no.~5, 553--618.


\bibitem{O93}
 A.~Y. Ol'shanskii, \textit{On residualing homomorphisms and $G$-subgroups of hyperbolic groups}, Internat. J. Algebra Comput. \textbf{3} (1993), no.~4, 365--409. 


\bibitem{Sledd}
L.~Sledd,
\textit{Assouad--Nagata dimension of finitely generated groups},
Topology Appl. \textbf{299} (2021), 107634.

\bibitem{ThomasVelickovic}
S.~Thomas and B.~Velickovic,
\textit{On the complexity of the isomorphism relation for finitely generated groups},
J. Algebra \textbf{217} (1999), no.~4, 352--373.

\bibitem{OOS}
A.~Yu. Ol'shanskii, D.~V. Osin, and M.~V. Sapir,
\textit{Lacunary hyperbolic groups},
Geom. Topol. \textbf{13} (2009), no.~4, 2051--2140.

\end{thebibliography}
\end{document}